\documentclass[11pt,reqno]{amsart}
\usepackage{enumerate}
\usepackage{amsfonts, amsmath, amssymb, amscd, amsthm, bm}
\usepackage{url}
\usepackage[linktocpage=true,colorlinks,citecolor=magenta,linkcolor=blue,urlcolor=magenta]{hyperref}
\usepackage{multicol}
\usepackage{comment}
\usepackage{graphicx}
\usepackage{pgfplots}
\usepackage{tikz}
\usetikzlibrary{arrows,shapes,trees,backgrounds}

 \newtheorem{thm}{Theorem}[section]

 \newtheorem{lem}[thm]{Lemma}
 
 \newtheorem{prop}[thm]{Proposition}
 \theoremstyle{definition}

 \newtheorem*{ack}{Acknowledgments}

 \numberwithin{equation}{section}
\def\ra{\rightarrow}
\def\pt{\partial}
\def\R{\mathbb{R}}

\begin{document}
\title[Minkowski-type inequality in Schwarzschild space]{On the Minkowski-type inequality for outward minimizing hypersurfaces in Schwarzschild space}
\author{Yong Wei}
\address{Mathematical Sciences Institute,
Australian National University,
ACT 2601 Australia}
\email{\href{mailto:yong.wei@anu.edu.au}{yong.wei@anu.edu.au}}

%\date{\today}
%\thanks {}
\subjclass[2010]{53C44, 53C42}
\keywords {Inverse mean curvature flow, Minkowski inequality, Outward minimizing, Schwarzschild space}

%%% ----------------------------------------------------------------------

\begin{abstract}
Using the weak solution of Inverse mean curvature flow, we prove the sharp Minkowski-type inequality for outward minimizing hypersurfaces in Schwarzschild space.
\end{abstract}

\maketitle
%\tableofcontents

\section{Introduction}\label{sec:1}

The Schwarzschild space is an $n$-dimensinal ($n\geq 3$) manifold $(M^n,g)$ with boundary $\pt M$, which is conformal to $\mathbb{R}^n\setminus D_{r_0}$, with the metric
\begin{equation}\label{s1:schw-1}
  g_{ij}(x)=~\left(1+\frac m2|x|^{2-n}\right)^{\frac 4{n-2}}\delta_{ij},\qquad x\in  \mathbb{R}^n\setminus D_{r_0},
\end{equation}
where $m>0$ is a constant, $r_0=(\frac m2)^{\frac 1{n-2}}$. The coordinate sphere $S_{r_0}=\partial D_{r_0}$ is the horizon of the Schwarzschild space and is outward minimizing. Equivalently, $M=[s_0,\infty)\times\mathbb{S}^{n-1}$ and
\begin{equation}\label{schwarz}
    g=\frac 1{1-2ms^{2-n}}ds^2+s^2g_{\mathbb{S}^{n-1}},
\end{equation}
where $s_0$ is the unique positive solution of $1-2ms_0^{2-n}=0$ and $g_{\mathbb{S}^{n-1}}$ is the canonical round metric on the unit sphere $\mathbb{S}^{n-1}$.  In this paper, we will denote
\begin{equation}\label{s1:f-def}
    f(x)~=~\sqrt{1-2ms^{2-n}},\qquad \mathrm{for~ any}~~ x=(s,\theta)\in M^n,
\end{equation}
which is called the potential function of $(M^n,g)$. As is well known, the Schwarzschild space is asymptotically flat, and is static in the sense that the potential function $f$ satisfies
\begin{equation}\label{s1:static-eqn}
  \nabla^2f~=~f\mathrm{Ric},\quad \Delta f~=~0,
\end{equation}
where $\mathrm{Ric}$ is the Ricci tensor of $(M^n,g)$, $\nabla,\nabla^2$ and $\Delta$ are gradient, Hessian and Laplacian operator with respect to the metric $g$ on $M^n$. It can be easily checked that the spacetime metric $\hat{g}=-f^2dt^2+g$ on $M^n\times \mathbb{R}$ solves the vacuum Einstein equation. In particular, \eqref{s1:static-eqn} implies that $(M^n,g)$ has constant zero scalar curvature $R$.

Let $\Omega$ be a bounded domain with smooth boundary in $(M^n,g)$. Then there are two cases:
\begin{itemize}
  \item[(i)] $\Omega$ has only one boundary component $\Sigma=\partial\Omega$ and we say that $\Sigma$ is null-homologous;
  \item[(ii)] $\Omega$ has two boundary components $\partial\Omega=\Sigma\cup\partial M$ and we say that $\Sigma$ is homologous to the horizon $\pt M$ of the Schwarzchild space.
\end{itemize}
The boundary hypersurface $\Sigma$ is said to be outward minimizing if whenever $E$ is a domain containing $\Omega$ then $|\partial E|\geq |\partial\Omega|$. From the first variational formula for area functional, an outward minimizing hypersurface must be a  mean-convex hypersurface.

The main result of this paper is the following Minkowski-type inequality for outward minimizing hypersurface in Schwarzschild space.
\begin{thm}\label{main-thm}
Let $\Omega$ be a bounded domain with smooth and outward minimizing boundary in the Schwarzschild space $(M^n,g)$. Assume either
\begin{itemize}
  \item[(1)] $n<8$, or
  \item[(2)] $n\geq 8$ and $\Sigma=\pt\Omega\setminus\pt M$ is homologous to the horizon.
\end{itemize}
Then
\begin{equation}\label{main-inequ}
    \frac 1{(n-1)\omega_{n-1}}\int_{\Sigma}fH d\mu\geq~\left(\frac{|\Sigma|}{\omega_{n-1}}\right)^{\frac{n-2}{n-1}}-2m,
\end{equation}
where $\omega_{n-1}$ is the area of the unit sphere $\mathbb{S}^{n-1}\subset\R^n$, and $|\Sigma|$ is the area of $\Sigma$ with respect to the induced metric from $(M^n,g)$. Moreover, the equality holds in \eqref{main-inequ} if and only if $\Sigma$ is a slice $\{s\}\times \mathbb{S}^{n-1}$.
\end{thm}
\noindent For mean convex and star-shaped hypersurface in Schwarzschild space, the inequality \eqref{main-inequ} was obtained by Brendle-Hung-Wang \cite{BHW} as the limit case of their inequality in Anti-de Sitter-Schwarzschild space. Note that a star-shaped hypersurface must be homologous to the horizon of the Schwarzchild space. Our result does not require the hypersurface to be star-shaped. The inequality \eqref{main-inequ} is a natural generalization of the classical Minkowski inequality for convex hypersurface $\Sigma$ in $\mathbb{R}^n$, which states that
\begin{equation}\label{mink-R}
  \int_{\Sigma}H d\mu\geq (n-1)\omega_{n-1}^{\frac 1{n-1}}|\Sigma|^{\frac{n-2}{n-1}}.
\end{equation}
The inequality \eqref{mink-R} was originally proved for convex hypersurfaces using the theory of convex geometry and was proved recently by Guan-Li \cite{GL} for mean convex and star-shaped hypersurfaces using the smooth solution of inverse mean curvature flow (IMCF). Huisken recently applied the weak solution of IMCF in \cite{HI01} to show that the inequality \eqref{mink-R} also holds for outward minimizing hypersurfaces in $\mathbb{R}^n$ (see \cite{Hui09}). The proof of this result was also given by Freire-Schwartz \cite{Fre-Sch-2014}. By letting $m\ra 0$, the Schwarzschild metric reduces to the Euclidean metric $g=ds^2+s^2g_{\mathbb{S}^{n-1}}$ and the potential function $f$ approaches to $1$. Thus Theorem \ref{main-thm} generalizes the result of Huisken and Freire-Schwartz to that for outward-minimizing hypersurfaces in Schwarzchild space.

To prove Theorem \ref{main-thm}, we use the standard procedure in proving geometric inequalities using the hypersurface curvature flows (see e.g.,\cite{BHW,Fre-Sch-2014,GL,HI01}). We will employ the weak solution of IMCF, which was developed by Huisken-Ilmanen in \cite{HI01} and was applied to prove the Riemannian Penrose inequality for asymptotically flat $3$-manifold with nonnegative scalar curvature. The weak solution of IMCF has also been applied in many other problems, see for example \cite{Bray-M,BN,Fre-Sch-2014,Lee-Neves}.  In our case, if $\Sigma$ is homologous to the horizon, then starting from $\Sigma$ there exists the weak solution of IMCF which is given by the level sets $\Sigma_t=\pt\Omega_t=\pt\{u<t\}$ of a proper locally Lipschitz function $u:\Omega^c\to \mathbb{R}^+$, where $\Omega^c$ denotes the compliments of $\Omega$ in $M$. Each $\Sigma_t$ is $C^{1,\alpha}$ away from a closed singular set $Z$ of Hausdorff dimension at most $n-8$ and $\Sigma_t$ will become $C^{1,\alpha}$ close to a large coordinate sphere as $t\to\infty$. On each $\Sigma_t$ we define the following quantity
\begin{equation}\label{s1:Qt-def}
    Q(t)=|\Sigma_t|^{-\frac{n-2}{n-1}}\left(\int_{\Sigma_t}fHd\mu_t+2(n-1)m\omega_{n-1}\right),
\end{equation}
where $|\Sigma_t|$ is the area of $\Sigma_t$. $Q(t)$ is well-defined because each $\Sigma_t$ is $C^{1,\alpha}$ with small singular set, the weak mean curvature of $\Sigma_t$ can be defined as a locally $L^1$ function using the first variation formula for area.  We will prove that $Q(t)$ is monotone non-increasing along the weak solution of IMCF. If $\Sigma=\partial\Omega$ is null-homologous, we first fill-in the region $W$ bounded by the horizon $\pt M$ to obtain a new manifold $\tilde{M}$ and then run the weak IMCF in $\tilde{M}$ with initial condition $\Sigma$.  When the flow $\Sigma_t=\pt\Omega_t$ nearly touches the horizon $\partial M$, we jump to the strictly minimizing hull $F$ of the union $\Omega_t\cup W$. Assume that $n<8$, we show that
\begin{equation*}
  |\partial F|\geq~|\Sigma_t|,\qquad \int_{\Sigma_t}fHd\mu_t\geq~\int_{\partial F}fHd\mu,
\end{equation*}
which implies that $Q(t)$ does not increase during the jump. Then we restart the flow from $\partial F$.  The restriction $n<8$ on the dimension in this case is due to that we need to jump to the strictly minimizing hull $F$ before we restart the flow, and $\pt F$ is only known to be smooth, more precisely $C^{1,1}$ for $n<8$. In summary, under the assumption of Theorem \ref{main-thm} we can prove that the quantity $Q(t)$ is monotone decreasing in time along the weak solution of IMCF.

Once we have the monotonicity of $Q(t)$, the next step is to estimate the limit when $t\to\infty$. For this, we will use the property that the weak solution becomes $C^{1,\alpha}$ close to a large coordinate sphere as $t\to\infty$ as shown in \cite[\S 7]{HI01}. The estimate that we will prove is the following:
\begin{align}\label{s1:Qt-lim}
    \lim_{t\ra\infty}Q(t)=&~(n-1)\omega_{n-1}^{\frac 1{n-1}}.
\end{align}
Then the main inequality \eqref{main-inequ} follows immediately from the monotonicity of $Q(t)$ and the estimate \eqref{s1:Qt-lim} on its limit. To complete the proof of Theorem \ref{main-thm}, we need to show the rigidity of the inequality. If the equality holds in \eqref{main-inequ}, from the proof of the monotonicity in \S \ref{sec:thm1} we know that $\Sigma$ must be homologous to the horizon and $\Sigma_t$ is umbilic a.e. for almost all time. This can be used to show that $\Sigma$ is an Euclidean sphere if it is considered as a hypersurface in $\mathbb{R}^n\setminus D_{r_0}$ with respect to the Euclidean metric. The last step is to show that $\Sigma$ is a sphere centered at the origin. We will use the property that a hypersurface to be umbilic is invariant under the conformal change of the ambient metric and the totally umbilic of $\Sigma_t$ for almost all time.

The rest of this paper is organized as follows. In \S \ref{sec:2}, we review some properties of the weak solution of IMCF. For more detail, we refer the readers to Huisken-Ilmanen's original paper \cite{HI01}. In \S \ref{sec:smooth}, we show how to derive the monotonicity of $Q(t)$ in the case that the flow is smooth. In \S \ref{sec:thm1}, we use the approximation argument to show the monotonicity of $Q(t)$ under the weak IMCF. In the last section, we estimate the limit of $Q(t)$ as $t\to\infty$ and complete the proof of Theorem \ref{main-thm}.

\begin{ack}
The author would like to thank Ben Andrews, Gerhard Huisken, Pei-Ken Hung and Hojoo Lee for their suggestions and discussions, and Haizhong Li, Mu-Tao Wang for their interests and comments.  The author would also like to thank the referee for helpful comments.  The author was supported by Ben Andrews throughout his Australian Laureate Fellowship FL150100126 of the Australian Research Council.
\end{ack}

\section{Weak solution of IMCF}\label{sec:2}
Let $(M^n,g)$ be the Schwarzschild space. The classical solution of IMCF is a smooth family $x: \Sigma\times [0,T)\to M$ of hypersurfaces $\Sigma_t=x(\Sigma,t)$ satisfying
\begin{equation}\label{s2:IMCF1}
  \frac{\pt x}{\pt t}=~\frac 1H\nu,\qquad x\in\Sigma_t,
\end{equation}
where $H, \nu$ are the mean curvature and outward unit normal of $\Sigma_t$, respectively. If the initial hypersurface is star-shaped and strictly mean convex, the smooth solution of \eqref{s2:IMCF1} exists for all time $t\in [0,\infty)$, and the flow hypersurfaces $\Sigma_t$ converge to large coordinate sphere in exponentially fast, see \cite{Li-Wei-Schwar,sch2017}. In general, without some special assumption on the initial hypersurface, the smoothness may not be preserved, the mean curvature may tend to zero at some points and the singularities develop. See for example the thin torus in Euclidean space (\cite[\S 1]{HI01}), i.e. the boundary of an $\epsilon$-neighborhood of a large round circle. The mean curvature is positive on this thin torus, so the smooth solution of \eqref{s2:IMCF1} exists for at least a short time. By deriving the upper bound of the mean curvature along the flow, we can see that the torus will steadily fatten up and the mean curvature will become negative in the donut hole in finite time.

In \cite{HI01}, Huisken-Ilmanen used the level-set approach and developed the weak solution of IMCF to overcome this problem. The evolving hypersurfaces are given by the level-sets of a scalar function $u:M^n\to \mathbb{R}$ via
\begin{equation*}
  \Sigma_t=\pt\{x\in M: u(x)<t\}.
\end{equation*}
Whenever $u$ is smooth with non vanishing gradient $\nabla u\neq 0$, the flow \eqref{s2:IMCF1} is equivalent to the following degenerate elliptic equation
\begin{equation}\label{s2:IMCF2}
  \mathrm{div}_M\left(\frac{\nabla u}{|\nabla u|}\right)=~|\nabla u|.
\end{equation}
Using the minimization principle and elliptic regularization, Huisken-Ilmanen proved the existence, uniquess, compactness and regularity properties of the weak solution of \eqref{s2:IMCF2}. The existence result only require mild growth assumption on the underlying manifold, and applies in particular to the Schwarzchild space here. We summaries their results in the following.

\begin{thm}[\cite{HI01}]\label{s2:thm}\label{s2:thm2-1}
Let $\Omega$ be a bounded domain with smooth boundary in the Schwarzschild space $(M^n,g)$ with $n<8$ and $\Sigma=\pt\Omega\setminus\pt M$. In case that $\Sigma$ is null-homologous, we fill-in the region $W$ bounded by the horizon. Then there exists a proper, locally Lipschitz function $u\geq 0$ on $\Omega^c=M\setminus \Omega$, called the weak solution of IMCF with initial condition $\Sigma$, satisfying
\begin{itemize}
  \item[(a)] $u|_{\Sigma}=0$, $\lim_{x\to\infty}u=\infty$. For $t>0$, $\Sigma_t=\partial\{u< t\}$ and $\Sigma_t'=\partial\{u>t\}$ define increasing families of $C^{1,\alpha}$ hypersurfaces.
  \item[(b)] The hypersurfaces $\Sigma_t$ (resp. $\Sigma_t'$) minimize (resp. strictly minimize) area among hypersurfaces homologous to $\Sigma_t$ in the region $\{u\geq t\}$. The hypersurface $\Sigma'=\partial \{u>0\}$ strictly minimizes area among hypersurfaces homologous to $\Sigma$ in $\Omega^c$.
  \item[(c)] For $t>0$, we have
\begin{equation}\label{s3:c}
  \Sigma_s\to\Sigma_t \quad \mathrm{as}\quad s\nearrow t,\qquad \Sigma_s\to\Sigma_t'\quad \mathrm{as}\quad s\searrow t
\end{equation}
locally in $C^{1,\beta}$ in $\Omega^c$, $\beta<\alpha$. The second convergence also holds as $s\searrow 0$.
  \item[(d)] For almost all $t>0$, the weak mean curvature of $\Sigma_t$ is defined and equals to $|\nabla u|$, which is positive and bounded for almost all $x\in \Sigma_t$.
  \item[(e)] For each $t>0$, $|\Sigma_t|=e^t|\Sigma'|$, and $|\Sigma_t|=e^t|\Sigma|$ if $\Sigma$ is outward minimizing.
\end{itemize}
For $n\geq 8$, the regularity and convergence are also true away from a closed singular set $Z$ of dimension at most $n-8$ and disjoint from $\bar{\Omega}$.
\end{thm}

Note that in \cite{Hei01}, Heidusch proved the optimal $C^{1,1}$ regularity for the level sets $\Sigma_t$ and $\Sigma_t'$ away from the singular set $Z$. The property (b) says that $\Omega_t=\{u<t\}$ and $\Omega_t'=\mathrm{int}\{u\leq t\}$ are minimizing hull and strictly minimizing hull in $\{u> t\}$. Here we call a set $E$ a minimizing hull in $G$ if $E$ minimizes area on the outside in $G$, that is, if
\begin{equation*}
  |\pt^*E\cap K|\leq~|\pt^*F\cap K|
\end{equation*}
for any $F$ of locally finite perimeter containing $E$ such that $F\setminus E\subset\subset G$, and any compact set $K$ containing $F\setminus E$. Here $\pt^*F$ denotes the reduced boundary of a set $F$ of locally finite perimeter. $E$ is called a strictly minimizing hull if equality implies that $F\cap G=E\cap G$. Define $E'$ to be the intersection of all strictly minimizing hulls in $G$ that contain $E$. Up to a set of measure zero, $E'$ may be realised by a countable intersection, so $E'$ itself is a strictly minimizing hull and open. We call $E'$ the strictly minimizing hull of $E$ in $G$.

The existence result of weak IMCF in Theorem \ref{s2:thm} was proved using a minimization principle (see \cite[\S 1]{HI01}), together with the elliptic regularization. Consider the following perturbed equation
\begin{equation}\label{s2:IMCF3}
\mathrm{div}_M\left(\frac{\nabla u^{\epsilon}}{\sqrt{|\nabla u^{\epsilon}|^2+\epsilon^2}}\right)=~\sqrt{|\nabla u^{\epsilon}|^2+\epsilon^2}
\end{equation}
on a large domain $\Omega_L=\{v<L\}$ defined using a subsolution $v$ of \eqref{s2:IMCF2}, with Dirichilet boundary condition $u^{\epsilon}=0$ on $\Sigma$ and $u^{\epsilon}=L-2$ on the boundary $\pt \Omega_L\setminus \Sigma$. This equation \eqref{s2:IMCF3} has the geometric interpretation that the downward translating graph
\begin{equation*}
  \hat{\Sigma}_t^{\epsilon}:=\mathrm{graph}\left(\frac{u^{\epsilon}(x)}{\epsilon}-\frac t{\epsilon}\right)
\end{equation*}
solves the smooth IMCF \eqref{s2:IMCF1} in the manifold $M\times \mathbb{R}$ of one dimension higher. Using the compactness theorem to pass the solutions of \eqref{s2:IMCF3} to limits as $\epsilon_i\to 0$, we obtain a family of cylinders in $M\times \mathbb{R}$, which sliced by $M\times \{0\}$ gives a family of hypersurfaces weakly solving \eqref{s2:IMCF2}. Similar techniques to show existence of weak solutions of geometric flows have been used by various authors, cf. \cite{ES,Il94,Moore,Mar,Sch08}.

From the argument in \cite[\S 3]{HI01}, we find that there exits a sequence of smooth function $u_i=u^{\epsilon_i}$ such that $u_i\to u$ locally uniformly in $\Omega^c$ to a function $u\in C^{0,1}(\Omega^c)$. $u_i$ and $u$ are uniformly bounded in $C^{0,1}(\Omega^c)$. For a.e. $t\geq 0$, the hypersurfaces $\hat{\Sigma}_t^i:=\hat{\Sigma}_t^{\epsilon_i}$ converges to  the cylinder $\hat{\Sigma}_t:=\Sigma_t\times \mathbb{R}$ locally in $C^{1,\alpha}$ away from the singular set $Z\times \mathbb{R}$. Moreover, as in \cite[\S 5]{HI01}, the mean curvature $H_{\hat{\Sigma}_t^i}$ of $\hat{\Sigma}_t^i$ converges to the weak mean curvature $H_{\hat{\Sigma}_t}$ of the cylinder $\hat{\Sigma}_t$ locally in $L^2$ sense for a.e. $t\geq 0$. Precisely,
\begin{equation}\label{s2:H-conv1}
  \int_{\hat{\Sigma}_t^i}\phi H_{\hat{\Sigma}_t^i}^2~\to~\int_{\hat{\Sigma}_t}\phi H_{\hat{\Sigma}_t}^2, \quad \mathrm{a.e.},~ t\geq 0
\end{equation}
for any cut-off function $\phi\in C_c^0(\Omega^c\times \mathbb{R})$. The weak second fundamental form $A_{\hat{\Sigma}_t}$ exists on $\hat{\Sigma}_t$ in $L^2$ and the lower semicontinuity implies
\begin{equation}\label{s2:A-weak}
  \int_{\hat{\Sigma}_t}|A_{\hat{\Sigma}_t}|^2\leq ~\liminf_{i\to\infty}\int_{\hat{\Sigma}_t^i}|A_{\hat{\Sigma}_t^i}|^2<~\infty
\end{equation}
for a.e. $t\geq 0$. Slicing this families $\hat{\Sigma}_t^i$ by $M\times \{0\}$, we obtain $\Sigma_t^{i}=\hat{\Sigma}_t^i\cap (M\times \{0\})$ and $\Sigma_t=\hat{\Sigma}_t\cap (M\times \{0\})$. Since $\hat{\Sigma}_t^i$ solves the smooth IMCF, its mean curvature in $M\times \mathbb{R}$ is
\begin{equation}\label{s2:H0}
  H_{\hat{\Sigma}_t^i}=\mathrm{div}_M\left(\frac{\nabla u^{\epsilon_i}}{\sqrt{|\nabla u^{\epsilon_i}|^2+\epsilon_i^2}}\right)=~\sqrt{|\nabla u^{\epsilon_i}|^2+\epsilon_i^2}.
\end{equation}
The mean curvature of $\Sigma_t^{i}$ considered as a hypersurface in $M$ is
\begin{align}\label{s2:H1}
  H_{\Sigma_t^{i}}= &\mathrm{div}_M\left(\frac{\nabla u^{\epsilon_i}}{|\nabla u^{\epsilon_i}|}\right)\nonumber\displaybreak[0]\\
  =&H_{\hat{\Sigma}_t^i}\frac{\sqrt{|\nabla u^{\epsilon_i}|^2+\epsilon_i^2}}{|\nabla u^{\epsilon_i}|} +\frac{\nabla u^{\epsilon_i}}{\sqrt{|\nabla u^{\epsilon_i}|^2+\epsilon_i^2}}\cdot \nabla\left( \frac{\sqrt{|\nabla u^{\epsilon_i}|^2+\epsilon_i^2}}{|\nabla u^{\epsilon_i}|}\right)\nonumber\displaybreak[0]\\
  =&H_{\hat{\Sigma}_t^i}\frac{\sqrt{|\nabla u^{\epsilon_i}|^2+\epsilon_i^2}}{|\nabla u^{\epsilon_i}|}-\epsilon_i^2\frac{\nabla u^{\epsilon_i}\cdot \nabla H_{\hat{\Sigma}_t^i} }{|\nabla u^{\epsilon_i}|^3H_{\hat{\Sigma}_t^i}}.
\end{align}
Since the limit function $u$ of $u^{\epsilon_i}$ has $|\nabla u|>0$ a.e. on $\Sigma_t$, using the weak convergence of $\nabla H_{\hat{\Sigma}_t^i}/H_{\hat{\Sigma}_t^i}$ (as in the proof of Lemma 5.2 in \cite{HI01}), we have that the second term on the right hand side of \eqref{s2:H1} converges to zero locally in $L^2$ sense as $\epsilon_i\to 0$. Thus, the mean curvature  $H_{\Sigma_t^{i}}$ of the sliced hypersurface $\Sigma_t^{i}$ converges to the weak mean curvature $H_{\Sigma_t}$ of $\Sigma_t$ locally in $L^2$ sense for a.e. $t\geq 0$.

\section{The smooth case}\label{sec:smooth}

As we mentioned in \S \ref{sec:1}, the key step to prove Theorem \ref{main-thm} is to show the monotonicity of $Q(t)$ defined in \eqref{s1:Qt-def} along the weak IMCF. In this section, we firstly show how to derive the monotonicity of $Q(t)$ in the smooth case. Let $\Sigma_t$ be a smooth solution of the IMCF \eqref{s2:IMCF1}.  It's well known that the following evolution equations for the area form $d\mu$ and mean curvature $H$ of $\Sigma_t$ in $(M^n,g)$ hold.
\begin{lem}
\begin{align}
   \pt_td\mu_t=&d\mu_t,\label{evl-measure}\\
 % \pt_t\nu= & \frac 1{H^2}\nabla H,\label{evl-nu}\\
\pt_tH=&-\Delta_{\Sigma} \frac{ 1}{H}-\frac{1}H\left(|A|^2+Ric(\nu,\nu)\right)\label{evl-H}
\end{align}
\end{lem}

Employing the above two evolution equations, we can derive the monotonicity of $Q(t)$ in the smooth case.
\begin{thm}
Let $\Sigma_t$ be a smooth solution of the IMCF \eqref{s2:IMCF1}. For any $0<t_1<t_2<T$, if $\Sigma_t$ is homologous to the horizon for all $t\in [t_1,t_2]$, then
\begin{equation*}
  Q(t_2)\leq ~Q(t_1)
\end{equation*}
with equality holds if and only if each $\Sigma_t$ is totally umbilic for $t\in [t_1,t_2]$. If $\Sigma_t$ is null-homologous for all $t\in [t_1,t_2]$, then
\begin{equation*}
  \tilde{Q}(t_2)\leq ~\tilde{Q}(t_1)
\end{equation*}
with equality holds if and only if each $\Sigma_t$ is totally umbilic  for $t\in [t_1,t_2]$, and $Q(t)$ is strictly decreasing in time $t\in [t_1,t_2]$, where
\begin{equation*}
  \tilde{Q}(t):=|\Sigma_t|^{-\frac{n-2}{n-1}}\int_{\Sigma_t}fHd\mu_t.
\end{equation*}
\end{thm}
\proof
The case that $\Sigma_t$ is homologous to horizon has been treated in \cite{BHW,Li-Wei-Schwar}. For convenience of readers, we include the proof here.  Using the evolution equations \eqref{evl-measure}--\eqref{evl-H},
\begin{align}
  \frac d{dt}\int_{\Sigma_t}fHd\mu_t= & \int_{\Sigma_t}\left(\pt_tfH+f\pt_tH+fH\right)d\mu_t\nonumber \displaybreak[0]\\
  = & \int_{\Sigma_t}\biggl(\langle{\nabla}f,\nu\rangle-f\Delta_{\Sigma}\frac 1H-\frac fH\left(|A|^2+{Ric}(\nu,\nu)\right)+fH\biggr)d\mu_t\nonumber\displaybreak[0]\\
  \leq & \int_{\Sigma_t}\biggl(\langle{\nabla}f,\nu\rangle-\frac 1H(\Delta_{\Sigma} f+f{Ric}(\nu,\nu))+\frac{n-2}{n-1}fH\biggr)d\mu_t,\label{eq4-2}
\end{align}
where we used $|A|^2\geq H^2/{(n-1)}$ in the last inequality. Combining the identity
\begin{equation*}
  \Delta_{\Sigma} f=\Delta f-\nabla^2f(\nu,\nu)-H\nu\cdot \nabla f
\end{equation*}
and the static equation \eqref{s1:static-eqn}, we have
\begin{equation}\label{eq3-0}
  \Delta_{\Sigma} f+f\mathrm{Ric}(\nu,\nu)=-H\nu\cdot \nabla f.
\end{equation}
Substituting \eqref{eq3-0} into \eqref{eq4-2} yields that
\begin{align}\label{eq4-1}
    \frac d{dt}\int_{\Sigma_t}fHd\mu_t\leq&\int_{\Sigma_t}\left(\frac{n-2}{n-1}fH+2\langle{\nabla}f,\nu\rangle\right)d\mu_t.
\end{align}
If equality holds in \eqref{eq4-1}, then $|A|^2=H^2/{(n-1)}$ and $\Sigma_t$ is totally umbilical.

If $\Sigma_t$ is homologous to the horizon for all $t\in [t_1,t_2]$, then denote $\Omega_t$ denote the region bounded by $\Sigma_t$ and the horizon $\pt M$. Applying the divergence theorem and noting that ${\Delta}f=0$ on $M$, we get
\begin{align*}
    \int_{\Sigma_t}\langle{\nabla}f,\nu\rangle d\mu_t=&\int_{\Omega_t}{\Delta}f +\int_{\pt M}\nabla f\cdot \nu_{\pt M}=m(n-2)\omega_{n-1}
\end{align*}
which is a constant. Thus we obtain
\begin{align}\label{eq3-1}
    \frac d{dt}\left(\int_{\Sigma_t}fHd\mu_t+2(n-1)m\omega_{n-1}\right)\leq&~\frac{n-2}{n-1}\left(\int_{\Sigma_t}fHd\mu_t+2(n-1)m\omega_{n-1}\right).
\end{align}
If $\Sigma_t=\pt \Omega_t$ is null-homologous for all $t\in [t_1,t_2]$, we have
\begin{align*}
    \int_{\Sigma_t}\langle{\nabla}f,\nu\rangle d\mu_t=&\int_{\Omega_t}{\Delta}f =0.
\end{align*}
Then
\begin{align}\label{eq3-2}
    \frac d{dt}\int_{\Sigma_t}fHd\mu_t\leq&~\frac{n-2}{n-1}\int_{\Sigma_t}fHd\mu_t.
\end{align}
Thus the theorem follows directly from \eqref{eq3-1}--\eqref{eq3-2} and the evolution equation of the area $|\Sigma_t|$
\begin{equation*}
  \frac d{dt}|\Sigma_t|=|\Sigma_t|.
\end{equation*}
\endproof

\section{The monotonicity}\label{sec:thm1}

Firstly, we prove the following lemma which was inspired by Lemma A.1 of \cite{Fre-Sch-2014}.
\begin{lem}
Suppose that $\Omega$ is a smooth bounded domain in $(M^n,g)$ and $\Sigma=\pt\Omega\setminus\pt M$.  Let $u:\Omega^c\ra \mathbb{R}_+$ be a smooth proper function with $u|_{\Sigma}=0$. Let $t>0$, $\Omega_t=\{u\leq t\}$ and $\Phi: (0,t)\ra \mathbb{R}_+$ be Lipschitz and compactly supported in $(0,t)$. Then $\varphi=\Phi\circ u:\Omega_t\ra \mathbb{R}_+$ satisfies
\begin{equation}\label{s3:1}
  -\int_{\Omega_t}f\nabla\varphi\cdot \nu Hdv_g~=~\int_{\Omega_t}\varphi\left(2\nabla f\cdot\nu H+fH^2-f|A|^2\right)dv_g,
\end{equation}
where $\nu,H,A$ denote the unit outward normal, mean curvature and second fundamental form of the level sets of $u$, $\nabla$ be the gradient operator on $(M^n,g)$ and $\nabla\varphi\cdot\nu=g(\nabla\varphi,\nu)$.
\end{lem}
\proof
The Sard's theorem implies that the level set $\Sigma_s=\{x\in\Omega^c: u(x)=s\}$ is regular ($\nabla u\neq 0$ on $\Sigma_s$) for a.e. $s>0$. Let $U\subset\Omega^c$ be the open subset where $\nabla u\neq 0$. For any regular level set $\Sigma_s$ with outward unit normal $\nu$, in $\Sigma_s\cap U$ the variation vector field along $\Sigma_s$ is $\nabla u/{|\nabla u|^2}$ and $\nu=\nabla u/{|\nabla u|}$.  By the second variation formula for area, we have
\begin{equation}\label{s3:lem1-1}
  -\frac 1{|\nabla u|}\nu\cdot\nabla H~=~\Delta_{\Sigma_s}|\nabla u|^{-1}+\frac 1{|\nabla u|}\left(|A|^2+\mathrm{Ric}(\nu,\nu)\right)
\end{equation}
in $\Sigma_s\cap U$, where $\Delta_{\Sigma_s}$ denotes the Laplacian operator with respect to the induced metric on $\Sigma_s$. We multiply \eqref{s3:lem1-1} by $f$ and integrate over $\Sigma_s$.
\begin{align*}
  -\int_{\Sigma_s}\frac f{|\nabla u|}\nu\cdot\nabla Hd\mu_s~=&~\int_{\Sigma_s}f\Delta_{\Sigma_s}|\nabla u|^{-1}+\int_{\Sigma_s}\frac f{|\nabla u|}\left(|A|^2+\mathrm{Ric}(\nu,\nu)\right)d\mu_s\\
  =&~\int_{\Sigma_s}\frac 1{|\nabla u|}\left(\Delta_{\Sigma_s}f+f|A|^2+f\mathrm{Ric}(\nu,\nu)\right)d\mu_s,
\end{align*}
where we used the divergence theorem in the second equality. Applying the identity \eqref{eq3-0}, we obtain
\begin{align}\label{s3:lem1-2}
  -\int_{\Sigma_s}\frac f{|\nabla u|}\nu\cdot\nabla Hd\mu_s~=&~\int_{\Sigma_s}\frac 1{|\nabla u|}\left(-H\nu\cdot \nabla f+f|A|^2\right)d\mu_s.
\end{align}
As $\Sigma_s$ is regular for a.e. $s>0$, the coarea formula and \eqref{s3:lem1-2} imply that
\begin{align}\label{s3:lem1-3}
   \int_{\Omega_t}f\varphi\nu\cdot\nabla Hdv_g=& ~\int_0^t\Phi(s)\int_{\Sigma_s} \frac f{|\nabla u|}\nu\cdot\nabla Hd\mu_sds\nonumber\\
   =&~ \int_0^t\Phi(s)\int_{\Sigma_s}\frac 1{|\nabla u|}\left(H\nu\cdot \nabla f-f|A|^2\right)d\mu_sds\nonumber\\
   =&~\int_{\Omega_t}\varphi\left(H\nu\cdot \nabla f-f|A|^2\right)dv_g.
\end{align}
We firstly assume that $\Phi\in C^1$. Then in the open subset $U=\{x\in \Omega^c: |\nabla u|\neq 0\}$, we have
\begin{align}\label{s3:lem1-4}
  \mathrm{div}(f\varphi H\nu)=&~\varphi\nabla f\cdot\nu H+f\nabla\varphi\cdot\nu H+f\varphi \nu\cdot\nabla H+f\varphi H\mathrm{div } ~\nu\nonumber\\
  =&~\varphi\nabla f\cdot\nu H+f\Phi'\circ u|\nabla u|H+f\varphi \nu\cdot\nabla H+f\varphi H\mathrm{div } ~\nu,
\end{align}
where $\mathrm{div}$ is the divergence operator on $(M^n,g)$. Since $\varphi$ is compactly supported in $\Omega_t$, integrating \eqref{s3:lem1-4} yields that
\begin{align}\label{s3:lem1-5}
 -\int_{\Omega_t}f\nabla\varphi\cdot \nu Hdv_g= & ~-\int_{\Omega_t} \mathrm{div}(f\varphi H\nu)dv_g \nonumber\\
   &\qquad +\int_{\Omega_t}\left(\varphi\nabla f\cdot\nu H+f\varphi \nu\cdot\nabla H+f\varphi H\mathrm{div } ~\nu\right)dv_g\nonumber\\
   =&~\int_{\Omega_t}\varphi\left(2H\nabla f\cdot \nu-f|A|^2\right)dv_g+\int_{\Omega_t}f\varphi H\mathrm{div } ~\nu dv_g,
\end{align}
where we used the divergence theorem and \eqref{s3:lem1-3}.

We now deal with the last term in \eqref{s3:lem1-5}.
%On any regular level set $\Sigma_s$, recall that the mean curvature $H$ satisfies (see \cite[p.81]{LSim})
%\begin{equation}\label{H-weak}
%  \int_{\Sigma_s\cap W}\mathrm{div}_{\Sigma_s}Xd\mu_s~=~\int_{\Sigma_s\cap W}H\nu\cdot Xd\mu_s
%\end{equation}
%for every $X\in C^{1}_c(TM)$ compactly supported in the precompact open set $W\subset M$, and by approximation for every $X\in W^{1,2}_c(TM)$.
Since $\Sigma_s$ is regular for a.e. $s>0$, the co-area formula and the first variation formula for area imply that
\begin{align}\label{s3:lem1-6}
  \int_{\Omega_t}f\varphi H\mathrm{div }(\nu) dv_g =& ~\int_0^t\int_{\Sigma_s}\frac{f\varphi H}{|\nabla u|}\mathrm{div}(\nu) d\mu_sds \nonumber\displaybreak[0]\\
   =&~ \int_0^t\int_{\Sigma_s}\frac{f\varphi H}{|\nabla u|}\mathrm{div}_{\Sigma_s}\nu d\mu_sds \nonumber\displaybreak[0]\\
   =&~\int_0^t\int_{\Sigma_s}\mathrm{div}_{\Sigma_s}\left(\frac{f\varphi H}{|\nabla u|}\nu\right) d\mu_sds \nonumber\\
   =&~\int_0^t\int_{\Sigma_s}\frac{f\varphi H^2}{|\nabla u|}d\mu_sds\nonumber\displaybreak[0]\\
   =&~\int_{\Omega_t}f\varphi H^2dv_g,
\end{align}
where in the second equality we used the fact that $\mathrm{div}(\nu)=\mathrm{div}_{\Sigma_s}\nu$ on $\Sigma_s\cap U$. Substituting \eqref{s3:lem1-6} into \eqref{s3:lem1-5} yields that
\begin{equation}\label{s3:lem1-7}
   -\int_{\Omega_t}f\nabla\varphi\cdot \nu Hdv_g~=~\int_{\Omega_t}\varphi\left(2\nabla f\cdot\nu H+fH^2-f|A|^2\right)dv_g
\end{equation}
for $\Phi\in C^1$. Since Lipschitz function can be approximated by $C^1$ function up to a set of measure zero (see \cite[p.32]{LSim}), we conclude that \eqref{s3:lem1-7} also holds for Lipschitz function $\Phi$ by approximation.
\endproof

\subsection{The case that $\Sigma$ is homologous to horizon}
\begin{lem}\label{prop-mono}
Let $\Omega$ be a smooth bounded domain in the Schwarzschild space $(M^n,g)$. Suppose that the boundary $\pt\Omega=\Sigma\cup\pt M$ and $\Sigma$ is outward minimizing. Let $\{\Sigma_t\}$ be the weak solution of IMCF in $\Omega^c=M\setminus\Omega$ with initial data $\Sigma$. Then for all $0<\bar{t}<t$,
\begin{align}\label{s3:monot-1}
  \int_{\Sigma_t}fHd\mu_t~\leq&~\int_{\Sigma_{\bar{t}}}fHd\mu+\frac{n-2}{n-1}\int_{\bar{t}}^t\left(\int_{\Sigma_s}fHd\mu_s+2(n-1)m\omega_{n-1}\right)ds
\end{align}
\end{lem}
\proof
As in the discussion in \S \ref{sec:2}, the weak solution of IMCF $u\in C^{0,1}_{\mathrm{loc}}(\Omega^c)$ can be approximated by smooth proper functions $u_i$ locally uniformly in $\Omega^c$, with $C^{1,\alpha}$ convergence of the level sets $\Sigma_s^i$ away from the singular set $Z$ and $L^2$ convergence of the weak mean curvature $H_i:=H_{\Sigma_s^i}$ of level sets for a.e. $s>0$. Moreover, we can show that $H_i$ converges to the mean curvature $H$ of the weak solution $\Sigma_s$ of IMCF in locally $L^2$ sense in any domain $\Omega_t$. In fact, by the coarea formula and \eqref{s2:H0}--\eqref{s2:H1},
\begin{align*}
  \int_{\Omega_t^i}\phi H_i =& \int_0^t\int_{\Sigma_s^i}\frac{\phi H_i}{|\nabla u_i|}d\mu_s^ids \\
   =& \int_0^t\int_{\Sigma_s^i}\phi \left(\frac{|\nabla u_i|^2+\epsilon_i^2}{|\nabla u_i|^2}-\epsilon_i^2\frac{\nabla u_i\cdot \nabla H_{\hat{\Sigma}_s^i} }{|\nabla u_i|^4H_{\hat{\Sigma}_s^i}}\right)d\mu_s^ids,
\end{align*}
where $\phi\in C_c^0(\Omega_t)$, $\Omega_t^i=\{x\in\Omega^c: u_i(x)<t\}$ and $\Omega_t=\{x\in\Omega^c: u(x)<t\}$. By the fact that $|\nabla u|>0$ a.e. on $\Sigma_s$ and the weak convergence of $\nabla H_{\hat{\Sigma}_s^i}/H_{\hat{\Sigma}_s^i}$, we have that $\int_{\Omega_t^i}\phi H_i\to \int_{\Omega_t}\phi H$ as $i\to\infty$. Similarly we have the convergence of $\int_{\Omega_t^i}\phi H_i^2\to \int_{\Omega_t}\phi H^2$.

For any nonnegative Lipschitz function $\Phi\in \mathrm{Lip}(0,t)$ with compact support in $(0,t)$ and $\varphi_i=\Phi\circ u_i$, by \eqref{s3:1} we have
\begin{align*}
  -\int_{\Omega_t}f\nabla\varphi_i\cdot \nu_i H_idv_g~=&~\int_{\Omega_t}\varphi_i\left(2\nabla f\cdot\nu_i H_i+fH_i^2-f|A_i|^2\right)dv_g\\
  \leq &~\int_{\Omega_t}\varphi_i\left(2\nabla f\cdot\nu_i H_i+\frac{n-2}{n-1}fH_i^2\right)dv_g,
\end{align*}
where we used $|A_i|^2\geq H_i^2/{(n-1)}$ in the last inequality. Taking the limit of $i\to\infty$, and using the convergence of $u_i$ and  $H_i$, we obtain that
\begin{align}\label{s3:lem2-1}
  -\int_{\Omega_t}f\nabla\varphi\cdot \nu Hdv_g~ \leq &~\int_{\Omega_t}\varphi\left(2\nabla f\cdot\nu H+\frac{n-2}{n-1}fH^2\right)dv_g
\end{align}
As $u$ is the weak solution of IMCF, we have $H=|\nabla u|$ a.e. in $\Omega_t$.  Also note that by Rademacher's theorem, the Lipschitz function $\Phi$ is differentiable a.e. in $(0,t)$. From the coarea formula and \eqref{s3:lem2-1}, we have
\begin{align}\label{s3:lem2-2}
  -\int_0^t\Phi'(s)\int_{\Sigma_s}fHd\mu_sds =& -\int_{\Omega_t}\Phi'(s)\nabla u\cdot\nu fHdv_g=~-\int_{\Omega_t}f\nabla\varphi\cdot \nu Hdv_g \nonumber\\
   \leq & ~\int_{\Omega_t}\varphi\left(2\nabla f\cdot\nu H+\frac{n-2}{n-1}fH^2\right)dv_g\nonumber\\
   =&~\int_{\Omega_t}\varphi\left(2\nabla f\cdot\nu +\frac{n-2}{n-1}fH\right)|\nabla u|dv_g\nonumber\\
   =&~\int_0^t\Phi(s)\int_{\Sigma_s}\left(2\nabla f\cdot\nu +\frac{n-2}{n-1}fH\right)d\mu_sds.
\end{align}
For any $0<\bar{t}<t$ and $0<\delta<(t-\bar{t})/2$, define $\Phi$ by
\begin{equation*}
  \Phi(s)~=~\left\{\begin{array}{ll}
                     0, ~& ~\mathrm{on}~[0,\bar{t}] \\
                     (s-\bar{t})/{\delta},~ & ~\mathrm{on}~[\bar{t},\bar{t}+\delta] \\
                     1,~ & ~\mathrm{on}~[\bar{t}+\delta,t-\delta] \\
                     (t-s)/{\delta},~ &~\mathrm{on}~ [t-\delta,t]
                   \end{array}\right.
\end{equation*}
The left hand side of \eqref{s3:lem2-2} is equal to
\begin{equation*}%\label{s3:lem2-0}
  \frac 1{\delta}\int_{t-\delta}^t\int_{\Sigma_s}fHd\mu_sds-\frac 1{\delta}\int_{\bar{t}}^{\bar{t}+\delta}\int_{\Sigma_s}fHd\mu_sds
\end{equation*}
Since for a.e. $s>0$, the level set $\Sigma_s^i$ of $u_i$ converges to $\Sigma_s$ in $C^{1,\alpha}$ away from the singular set $Z$ of Hausdorff dimension at most $n-8$ with $L^2$ convergence of the weak mean curvature, taking the limits $\delta\to 0$ in \eqref{s3:lem2-2}, we find that for a.e. $0<\bar{t}<t$
\begin{equation}\label{s3:lem2-3}
   \int_{\Sigma_t}fHd\mu_t-\int_{\Sigma_{\bar{t}}}fHd\mu_{\bar{t}}~\leq~\int_{\bar{t}}^t\int_{\Sigma_s}\left(2\nabla f\cdot\nu+\frac{n-2}{n-1}fH\right)d\mu_sds.
\end{equation}

To show \eqref{s3:lem2-3} holds for all pair of $0<\bar{t}<t$, we use the $C^{1,\beta}$ convergence \eqref{s3:c} and the weak convergence of mean curvature. For any $t>0$, we can find a sequence of time $t_i\nearrow t$ such that $0<\bar{t}<t_i$ satisfies \eqref{s3:lem2-3}, then $\Sigma_{t_i}\to \Sigma_t$ in $C^{1,\beta}$ away from the singular set $Z$ as $i\to\infty$ by \eqref{s3:c}. As the weak mean curvature of $\Sigma_{t_i}$ equals to $|\nabla u|$ a.e. and is uniformly bounded for a.e. $x\in\Sigma_{t_i}$, it follows from the Riesz Representation theorem that (see (1.13) in \cite{HI01})
\begin{equation}\label{s3:Mean-WC}
  \int_{\Sigma_{t_i}}H_{\Sigma_{t_i}}\nu_{\Sigma_{t_i}}\cdot X~\to~\int_{\Sigma_t}H_{\Sigma_t}\nu_{\Sigma_t}\cdot X,\quad X\in C_c^0(TM).
\end{equation}
Then
\begin{align}\label{rem-1}
  \int_{\Sigma_t}fHd\mu_t =& ~\lim_{i\to\infty}\int_{\Sigma_{t_i}}fH_{t_i}d\mu_{t_i}
 % \nonumber \\
%  \leq & ~\int_{\Sigma_{\bar{t}}}fHd\mu+\int_{\bar{t}}^t\int_{\Sigma_s}\left(2\nabla f\cdot\nu+\frac{n-2}{n-1}fH\right)d\mu_sds
\end{align}
and \eqref{s3:lem2-3} holds for all $t>0$ and a.e. $\bar{t}>0$ with $\bar{t}<t$. Similarly for any $\bar{t}$ with $0<\bar{t}<t$, we can find a sequence of time $\bar{t}_i\searrow \bar{t}$ such that $0<\bar{t}_i<t$ satisfies \eqref{s3:lem2-3}. By the convergence \eqref{s3:c} and \eqref{s3:Mean-WC}, we have
\begin{equation}\label{s3:tb-1}
 \int_{\Sigma_{\bar{t}}'}fHd\mu_{\bar{t}}'~=~\lim_{i\to\infty} \int_{\Sigma_{\bar{t}_i}}fHd\mu_{\bar{t}_i}.
\end{equation}
Recall that Heidusch \cite{Hei01} proved the optimal local $C^{1,1}_{\mathrm{loc}}$ regularity for the level sets $\Sigma_{\bar{t}}$ and $\Sigma_{\bar{t}}'$ away from the singular set $Z$. By \cite[(1.15)]{HI01} the weak mean curvature $H$ of $\Sigma_{\bar{t}}$ and $\Sigma_{\bar{t}}'$ satisfy
\begin{equation*}
  H=0~~\mathrm{ on}~ \Sigma_{\bar{t}}'\setminus (\Sigma_{\bar{t}}\cup Z), \quad H_{\Sigma_{\bar{t}}'}=H_{\Sigma_{\bar{t}}}\geq 0, ~\mathrm{a.e.} ~\mathrm{on}~~ \Sigma_{\bar{t}}'\cap \Sigma_{\bar{t}}.
\end{equation*}
As the weak mean curvature $H$ is nonnegative on $\Sigma_{\bar{t}}$, we deduce that
\begin{equation}\label{s3:tb-2}
  \int_{\Sigma_{\bar{t}}'}fHd\mu_{\bar{t}'}~\leq ~\int_{\Sigma_{\bar{t}}}fHd\mu_{\bar{t}}.
\end{equation}
Thus by \eqref{rem-1}--\eqref{s3:tb-2}, we conclude that \eqref{s3:lem2-3} holds for all pair of $0<\bar{t}<t$. Since $\Sigma$ is assumed to be outward minimizing, \eqref{s3:lem2-3} is also true for $\bar{t}=0$.

Finally, for the first integral on the right hand side of \eqref{s3:lem2-3}, using the divergence theorem and noting that $\Delta f=0$ on $M$, we get
\begin{align}\label{s3:lem2-7}
    \int_{\Sigma_s}\nabla f\cdot \nu =&\int_{\Omega_s\cup\Omega}{\Delta}f +\int_{\pt M}\nabla f\cdot \nu_{\pt M} =~m(n-2)\omega_{n-1},
\end{align}
by first computing on $\Sigma_s^i$ and then passing to limits. Inserting \eqref{s3:lem2-7} into \eqref{s3:lem2-3}, we obtain the inequality \eqref{s3:monot-1} for all pair of $0<\bar{t}<t$.
\endproof

\begin{prop}\label{prop-1}
Under the assumption of Lemma \ref{prop-mono}, the quantity $Q(t)$ is monotone non-increasing for all $t$. Moreover, if $Q(t)=Q(\bar{t})$ for some pair $0<\bar{t}<t$, we have that $\Sigma_s$ is umbilic a.e. for a.e. $s\in [\bar{t},t]$.
\end{prop}
\proof
By Gronwall's lemma, \eqref{s3:monot-1} implies that
\begin{align*}%\label{prop-1-1}
\int_{\Sigma_t}fHd\mu_t+2(n-1)m\omega_{n-1} \leq & ~\left(\int_{\Sigma_{\bar{t}}}fHd\mu+2(n-1)m\omega_{n-1}\right)e^{\frac{n-2}{n-1}(t-\bar{t})}
\end{align*}
for all $0<\bar{t}<t$. Since $\Sigma$ is outward minimizing, Theorem \ref{s2:thm2-1} implies that $|\Sigma_t|=e^t|\Sigma|$ for all $t\geq 0$. Then the quantity $Q(t)$ is monotone non-increasing for all $t\geq 0$.  If $Q(t)=Q(\bar{t})$ for some pair $0<\bar{t}<t$, from the proof of Lemma \ref{prop-mono}, we have that $H^2=(n-1)|A|^2$ a.e. on $\Sigma_s$ for a.e. $s\in [\bar{t},t]$.
\endproof

\subsection{The case that $\Sigma$ is null homologous}\label{sec:4-2}
Now we consider the case that the bounded domain $\Omega$ has boundary $\pt\Omega=\Sigma$, which is null-homologous and outward minimizing. By the argument in \cite[\S 6]{HI01}, we fill-in the region $W$ bounded by the horizon $\pt M$ to obtain a new space $\tilde{M}$, and then run the weak IMCF in $\tilde{M}$ with initial condition $\Sigma$, except that when the flow $\Sigma_t$ is nearly entering the filled-in region $W$, we jump to a strictly minimizing hull $F$ enclosing $\Omega_t\cup W$. Then we restart the flow from $\partial F$.
\begin{center}
\begin{tikzpicture}[ thick, fill opacity=0.8]
\filldraw[fill=gray!10,style=thick]
(-1.8,-1.8) [rounded corners=5pt]-- (-2.1,-1.2)
 [rounded corners=15pt]--(-1.8,-0.2)--(-1.7,0.9)--(-0.7,1.7)
 [rounded corners=5pt]--(1,3)--(1.7,3.3)--(2.3,3.5)--(2.9,3.4)
  [rounded corners=15pt] --(3.4,2.5)--(3.2,1.6)--(3.1,0.2)
  [rounded corners=5pt] --(3.3,-0.6)
   [rounded corners=15pt] --(2.8,-1.3)
   [rounded corners=5pt]--(0.4,-1.6)--(-1,-2)
    [rounded corners=15pt] --cycle
   (-0.5,-0.5) [rounded corners=5pt]-- (-0.5,0.4)
 [rounded corners=15pt]--(-0.7,1.7)
 [rounded corners=5pt]--(1,3)--(1.7,3.3)--(2.3,3.5)--(2.9,3.4)
  [rounded corners=15pt] --(3.4,2.5)--(3.2,1.6)--(3.1,0.2)
  [rounded corners=5pt] --(3.3,-0.6)
   [rounded corners=15pt] --(2.8,-1.3)
   [rounded corners=5pt]--(0.4,-1.6)--(0,-1)
    [rounded corners=15pt] --cycle;

\filldraw[color=black!60,fill=red!5,style=thick]
(0,0)[rounded corners=10pt] -- (0.2,0.9) -- (0.8,2.4)--(1.7,2.6)
 [rounded corners=5pt]--(2.9,2)--(2.6,1.4)
 [rounded corners=10pt]--(3,0.4)--(2.1,0)--(1.1,-1.0)--cycle
(-0.5,-0.5) [rounded corners=5pt]-- (-0.5,0.4)
 [rounded corners=15pt]--(-0.7,1.7)
 [rounded corners=5pt]--(1,3)--(1.7,3.3)--(2.3,3.5)--(2.9,3.4)
  [rounded corners=15pt] --(3.4,2.5)--(3.2,1.6)--(3.1,0.2)
  [rounded corners=5pt] --(3.3,-0.6)
   [rounded corners=15pt] --(2.8,-1.3)
   [rounded corners=5pt]--(0.4,-1.6)--(0,-1)
    [rounded corners=15pt] --cycle;

\filldraw[color=red!60, fill=blue!10, thick](-1,-1) circle (0.7);

 \filldraw[color=black!60,fill=green!10,style=thick]
(0,0)[rounded corners=10pt] -- (0.2,0.9) -- (0.8,2.4)--(1.7,2.6)
 [rounded corners=5pt]--(2.9,2)--(2.6,1.4)
 [rounded corners=10pt]--(3,0.4)--(2.1,0)--(1.1,-1.0)--cycle;

   \node  at (2,-1.1) {$\Omega_{t_1}$};
    \node  at (-1.3,0) {$F$};
     \node  at (-1,-1) {$W$};
     \node  at (1.5,1.4) {$\Omega$};
     \node at (-0.6,0.4) {$\Sigma_{t_1}$};
     \node at (0.05,0.8) {$\Sigma$};
     \node at (-1.8,0.9) {$\pt F$};
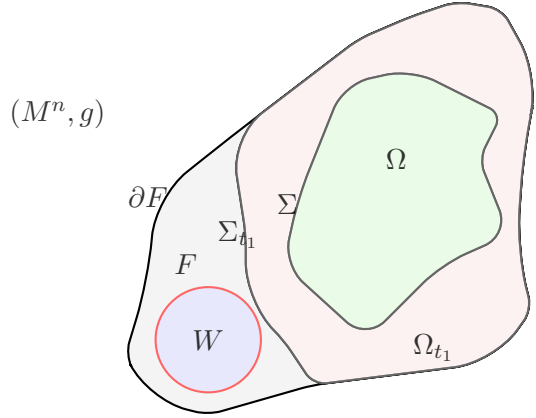
     \node at (-1,-2.5) {Figure 1: $F$ is strictly minimizing hull of $\Omega_{t_1}$ and $W$};
     \node at (-3, 2) {$(M^n,g)$};
\end{tikzpicture}
\end{center}
Suppose that $t_1$ is the jump time. Then before $t_1$, each $\Sigma_t$ is null homologous. Using the divergence theorem as in \eqref{s3:lem2-7}, we have that $\int_{\Sigma_t}\nabla f\cdot\nu=0$ if $t\leq t_1$. The similar argument as in Lemma \ref{prop-mono} and Proposition \eqref{prop-1} implies that
\begin{equation*}
  |\Sigma_t|^{-\frac{n-2}{n-1}}\int_{\Sigma_t}fHd\mu_t
\end{equation*}
is monotone non-increasing in time $t$ if $t\leq t_1$. As $|\Sigma_t|$ is increasing and the mass $m>0$, we have that $Q(t)$ is strictly decreasing for $t\leq t_1$.
\begin{lem}
$\pt F$ is $C^{1,\alpha}$ away from a singular set of Hausdorff dimension at most $n-8$.
\end{lem}
\proof
If $n<8$, then $\Sigma_{t_1}=\pt \Omega_{t_1}$ is $C^{1,\alpha}$ and $\pt M$ is smooth, which combined with the Regularity Theorem 1.3 (ii) of \cite{HI01} implies that $\pt F$ is $C^{1,\alpha}$. If $n\geq 8$, then $\Sigma_{t_1}$ has singular set of Hausdorff dimension at most $n-8$. By the variational formulation of the weak solution of IMCF described in \cite[\S 1]{HI01}, $\Omega_{t_1}=\{u<t_1\}$ minimizes $J_u$ in $\Omega^c$ among sets of locally finite perimeter $F$ with $F\Delta E\subset\subset \Omega^c$, where
\begin{equation*}
  J_u(F)=J_u^K(F)=|\pt^*F\cap K|-\int_{F\cap K}|\nabla u|.
\end{equation*}
and  $K$ is any compact set containing $F\Delta E$. Since $|\nabla u|$ is bounded above locally uniformly by Theorem 3.1 of \cite{HI01}, the obstacle $\Omega_{t_1}\cup W$ satisfies the assumption of the main theorem in \cite{Ba-M82} (see also Proposition 2 of \cite{Ta82} ). Therefore, as the strictly minimizing hull of $\Omega_{t_1}\cup W$, $F$ has boundary $\pt F$ which is $C^{1,\alpha}$ away from a singular set of Hausdorff dimension at most $n-8$.
\endproof

Since $\Sigma_{t_1}$ is outward minimizing and $F\supset \Omega_{t_1}$, we have
\begin{equation}\label{jump1}
  |\pt F|\geq ~|\Sigma_{t_1}|.
\end{equation}
Now we assume that $n<8$.  As $\Sigma_{t_1}$ is $C^{1,\alpha}$ with nonnegative bounded weak mean curvature, the standard Calderon-Zygmund estimate implies that $\Sigma_{t_1}$ is of class $W^{2,p}$ for all $1\leq p<\infty$. We can choose a sequences of sets $E_i$ containing $\Omega_{{t_1}}$ by mollification such that $\pt E_i$ is smooth and converges to $\Sigma_{t_1}=\pt\Omega_{{t}_1}$ in $C^{1,\alpha}\cap W^{2,p}$. The Regularity Theorem 1.3 and (1.15) of \cite{HI01} imply that $\pt (E_i\cup W)'$ is $C^{1,1}$ and $H=0$ on $\pt (E_i\cup W)'\setminus \pt (E_i\cup W)$. Thus
\begin{equation*}
  \int_{\pt (E_i\cup W)'}fH=~\int_{\pt (E_i\cup W)'\cap \pt E_i}fH=~\int_{\pt E_i}fH-\int_{\pt E_i\setminus\pt (E_i\cup W)'}fH
\end{equation*}
It can be seen that $(E_i\cup W)'\to F$ and $\pt (E_i\cup W)'\to \pt F$ in $C^{1,\alpha}$. Passing to limits and using the nonnegativity of the weak mean curvature on $\Sigma_{t_1}$, we obtain
\begin{equation}\label{jump2}
  \int_{\pt F}fH\leq ~\int_{\Sigma_{t_1}}fH.
\end{equation}
The estimates \eqref{jump1}--\eqref{jump2} imply that the quantity $Q(t)$ defined in \eqref{s1:Qt-def} does not increase during the jump. Similar as in \cite[\S 6]{HI01}, we can show that $F$ is a suitable initial condition to restart the flow. We approximate $\Sigma_{t_1}$ in $C^1 $ by a sequence of smooth hypersurfaces $\pt U_i$ with uniformly bounded mean curvature and $U_i\supset\Omega_{t_1}$, using Lemma 6.2 of \cite{HI01}. Then by Regularity Theorem 1.3 and (1.15) of \cite{HI01}, $\pt (U_i\cup W)'$ is $C^{1,1}$ and has uniformly bounded mean curvature as well. It can be checked that $\pt (U_i\cup W)'$ converges to $\pt F$ in $C^1$. Then slightly mollifying $\pt (U_i\cup W)'$ shows that $\pt F$ is approximated in $C^1$ by smooth hypersurfaces $\pt F^i$ with uniformly bounded mean curvature. The Existence Theorem 3.1 of \cite{HI01} gives a solution $u^i$ of weak IMCF $F_t^i, t>0$ with initial condition $F^i$, and uniform bounds on the gradient $|\nabla u^i|$. Passing to the limits and applying the Compactness Theorem 2.1 of \cite{HI01}, we obtain the solution $u$ of the  weak IMCF with initial condition $F$.  The proof of Proposition \ref{prop-1} and the argument in \cite[p.407]{HI01} imply that $Q(t)$ is monotone non-increasing in $t$ for $t\geq 0$ along the weak IMCF $F_t$ with the initial condition $F$. Thus we conclude that $Q(t)$ is monotone non-increasing for all time $t$.

\begin{prop}\label{prop-1}
Let $n<8$ and $\Omega$ be a bounded domain with smooth boundary $\pt\Omega=\Sigma$ in the Schwarzschild space $(M^n,g)$. Assume that $\Sigma$ is outward minimizing. Then $Q(t)$ is monotone non-increasing for all time along the weak IMCF.
\end{prop}

\section{Proof of the main theorem}
In \S \ref{sec:thm1}, we proved that the quantity $Q(t)$ is monotone non-increasing along the weak IMCF. In this section, we first estimate the limit of $Q(t)$ as $t\ra\infty$.
\begin{prop}\label{prop-2}
We have
\begin{align}\label{Qt-lim}
    \lim_{t\ra\infty}Q(t)~= &~(n-1)\omega_{n-1}^{\frac 1{n-1}}.
\end{align}
\end{prop}
\proof
Denote by $U=\mathbb{R}^n\setminus D_{r_0}$ the asymptotic flat end of $M$. The Schwarzschild metric on $U$  is
\begin{equation*}%\label{s4:schw-1}
  g=~\left(1+\frac m2r^{2-n}\right)^{\frac 4{n-2}}\left(dr^2+r^2g_{\mathbb{S}^{n-1}}\right),
\end{equation*}
For any $\lambda>0$, define the blow down object by
\begin{equation*}
  \Sigma_t^{\lambda}:=\lambda \Sigma_t=\{\lambda x:x\in \Sigma_t\},\qquad g^{\lambda}(x):=\lambda^2g(x/{\lambda}).
\end{equation*}
Let $r(t)$ be such that $|\Sigma_t|=\omega_{n-1}r(t)^{n-1}$. Then $|\Sigma_t^{1/{r(t)}}|_{g^{1/{r(t)}}}=\omega_{n-1}$ and the blow down Lemma 7.1 of \cite{HI01} implies that \begin{equation}\label{s4:blowd-1}
  \Sigma_t^{1/{r(t)}}~\to~ \partial D_1
\end{equation}
in $C^{1,\alpha}$ as $t\to \infty$. As in the proof of Lemma 7.1 of \cite{HI01}, there exist constants $C,R_0>0$ depending only on the dimension $n$ such that
\begin{equation*}
  |\nabla u(x)|~\leq ~\frac C{|x|},\quad\mathrm{ for~ all }~~|x|\geq R_0.
\end{equation*}
By the property (d) of the weak solution of IMCF, we have
\begin{equation}%\label{}
  |H|=|\nabla u|~\leq ~\frac C{|x|}\leq~\frac C{r(t)}, \quad \mathrm{a.e., on}~~ \Sigma_t
\end{equation}
for a.e. sufficiently large $t$, where we have used \eqref{s4:blowd-1} to relate $|x|$ to $r(t)$. The mean curvature of $\Sigma_t^{1/{r(t)}}$ with respect to the metric $g^{1/{r(t)}}$ satisfies
\begin{equation}\label{s4:H2}
  H^{1/{r(t)}}(x)=~r(t)H(r(t)x)\leq ~C,\quad \mathrm{a.e.}~~ x\in \Sigma_t^{1/{r(t)}}.
\end{equation}
for a.e. sufficiently large $t$. Write $\Sigma_t^{1/{r(t)}}$ as graphs of $C^{1,\alpha}$ functions over $\partial D_1$. By \eqref{s4:blowd-1} and \eqref{s4:H2}, for any sequence of time $t_i\to\infty$ such that \eqref{s4:H2} holds for time $t_i$, we have the weak convergence of the mean curvature
\begin{equation}\label{s4:H3}
  \int_{\Sigma_{t_i}^{1/{r(t_i)}}}H_{\Sigma_{t_i}^{1/{r(t_i)}}}\nu_{\Sigma_{t_i}^{1/{r(t_i)}}}\cdot X~\to~\int_{\pt D_1}H_{\partial D_1}\nu_{\pt D_1}\cdot X,\qquad X\in C_c^0(TM),
\end{equation}
Recall that
\begin{equation}\label{s4:f}
  f(x)=\sqrt{1-2ms(x)^{2-n}}=~1-m|x|^{2-n}+O(|x|^{4-2n}),
\end{equation}
where
\begin{equation*}
  s(x)=|x|\left(1+\frac m2|x|^{2-n}\right)^{\frac 2{n-2}},\qquad \forall~x\in U.
\end{equation*}
Then by \eqref{s4:blowd-1}, and \eqref{s4:H2} -- \eqref{s4:f}, we have that
\begin{align*}
  \lim_{t_i\to\infty}|\Sigma_{t_i}|^{-\frac{n-2}{n-1}}\int_{\Sigma_{t_i}}fHd\mu_{t_i} =&\omega_{n-1}^{-\frac{n-2}{n-1}}\lim_{t_i\to\infty}r(t_i)^{-(n-2)}\int_{\Sigma_{t_i}}fHd\mu_{t_i}  \\
  = & \omega_{n-1}^{-\frac{n-2}{n-1}}\lim_{t_i\to\infty}\int_{\Sigma_{t_i}^{1/{r(t_i)}}}f(r(t_i)x)H^{1/{r(t_i)}}(x)d\mu_{\Sigma_{t_i}^{1/{r(t_i)}}}  \\
  =&(n-1)\omega_{n-1}^{\frac 1{n-1}}.
\end{align*}
Observe that $2(n-1)m\omega_{n-1}$ is a fixed constant and $|\Sigma_{t_i}|$ goes to infinity as $t_i\to\infty$. Therefore
\begin{equation*}%\label{s4:Qt-lim}
   \lim_{t_i\to\infty}Q(t_i)~=~(n-1)\omega_{n-1}^{\frac 1{n-1}},
\end{equation*}
which combined with the monotonicity of $Q(t)$ yields the estimate \eqref{Qt-lim}.
\endproof

We now complete the proof of our main theorem.
\begin{proof}[Proof of Theorem \ref{main-thm}]
Proposition \ref{prop-2} together with the monotonicity of $Q(t)$ yields the inequality \eqref{main-inequ} in Theorem \ref{main-thm} immediately. To complete the proof of Theorem \ref{main-thm}, it remains to prove the rigidity of the inequality \eqref{main-inequ}.

If equality holds in \eqref{main-inequ} for $\Sigma$, then $Q(0)=Q(t)$ for all $t>0$. Then the initial hypersurfae $\Sigma$ must be homologous to the horizon, because if not, $Q(t)$ should be strictly decreasing during the jump as described in \S \ref{sec:4-2}. From the proof of Lemma \ref{prop-mono}, the fact that $Q(0)=Q(t)$ for all $t>0$ also implies $H^2=(n-1)|A|^2$ a.e. on $\Sigma_t$ for almost all time $t\geq 0$. Since $\Sigma$ is smooth and outward minimizing,  Theorem \ref{s2:thm2-1} implies that $\Sigma_t\to \Sigma$ locally in $C^{1,\beta}$ as $t\to 0+$. We can choose a sequence of time $t_i\to 0+$ with $\int_{\Sigma_{t_i}}|\mathring{A}|^2=0$ and $\Sigma_{t_i}$ converges to $\Sigma$ locally in $C^{1,\beta}$ as $t_i\to 0+$. The lower semicontinuity implies
 \begin{equation*}%\label{s5-umb-2}
   \int_{\Sigma}|\mathring{A}|^2\leq ~\liminf_{t_i\to 0+}\int_{\Sigma_{t_i}}|\mathring{A}|^2=~0.
 \end{equation*}
and then $\Sigma$ is totally umbilic in the Schwarzchild space $(M^n,g)=(\mathbb{R}^n\setminus D_{r_0}, g_{ij})$. Denote the Schwarzschild metric $g=e^{2\psi}\delta_{ij}$ on $\mathbb{R}^n\setminus D_{r_0}$, where
\begin{equation*}%\label{psi}
  \psi=~\frac 2{n-2}\ln\left(1+\frac m2r^{2-n}\right).
\end{equation*}
Let $\bar{\nu}, \bar{H}$ and $\bar{h}_i^j$ be the unit outward normal, mean curvature and shape operator of $\Sigma$ with respect to $(\mathbb{R}^n\setminus D_{r_0}, \delta_{ij})$. Then they satisfies the following transformation formula under the conformal change of the ambient metric:
\begin{align}
 e^{\psi}H=&~\bar{H}+d\psi\cdot \bar{\nu},\qquad \nu=~e^{-\psi}\bar{\nu},\label{s5:conf}\\
 \mathring{h}_i^j=&~h_i^j-\frac {H}{n-1}\delta_i^j~=~e^{-\psi}\left(\bar{h}_i^j-\frac {\bar{H}}{n-1}\delta_i^j\right)~=~e^{-\psi}\mathring{\bar{h}}_i^j,\label{s5:conf-1}
\end{align}
where $\mathring{h}_i^j$ and $\mathring{\bar{h}}_i^j$ are the trace-less second fundamental forms of $\Sigma$ in $\mathbb{R}^n\setminus D_{r_0}$ with respect to the metrics $\delta_{ij}$ and $g_{ij}$ respectively. Equation \eqref{s5:conf-1} says that the totally umbilicity of a hypersurface is invariant under the conformal change of the ambient metric. Then $\Sigma$ is totally umbilic in $\mathbb{R}^n\setminus D_{r_0}$ with respect to the Euclidean metric $\delta_{ij}$ and therefore is a sphere in $(\mathbb{R}^n\setminus D_{r_0}, \delta_{ij})$.

We next show that $\Sigma$ is a sphere centered at the origin in $(\mathbb{R}^n\setminus D_{r_0}, \delta_{ij})$, and is a slice $\{s\}\times \mathbb{S}^{n-1}$ if considered as a hypersurface in the Schwarzschild space.
Suppose that the radius of the sphere $\Sigma$ is $r$. Then the mean curvature of $\Sigma$ in $(\mathbb{R}^n\setminus D_{r_0}, \delta_{ij})$ is $\bar{H}=(n-1)/r$ and
\begin{equation}\label{s5:conf2}
   d\psi\cdot\bar{\nu}=~-\frac {mr^{1-n}}{1+\frac m2r^{2-n}}\pt_r\cdot \bar{\nu}\geq ~-\frac {mr^{1-n}}{1+\frac m2r^{2-n}}.
\end{equation}
This implies that the mean curvature of $\Sigma$ in the $(\mathbb{R}^n\setminus D_{r_0}, g_{ij})$ satisfies $H=e^{-\psi}(\bar{H}+d\psi\cdot \bar{\nu})>0$.  Since $\Sigma$ is strictly mean convex, starting from $\Sigma$ there exists a unique smooth solution to the IMCF \eqref{s2:IMCF1} in Schwarzschild space, which coincides with the weak solution for a short time $t\in [0,\delta)$ by the Smooth Start Lemma 2.4 of \cite{HI01}. Arguing similarly as before, each $\Sigma_t, t\in[0,\delta),$ is a sphere in $(\mathbb{R}^n\setminus D_{r_0}, \delta_{ij})$. By the conformal transformation formulas \eqref{s5:conf}, $\Sigma_t$ solves the following corresponding flow in Euclidean space $(\mathbb{R}^n\setminus D_{r_0}, \delta_{ij})$
\begin{equation}\label{Flow-E}
  \frac{\pt}{\pt t}X(x,t)=\frac 1{\bar{H}+d\psi\cdot \bar{\nu}}\bar{\nu}(x,t), \quad t\in [0,\delta).
\end{equation}
Under the flow \eqref{Flow-E}, the shape operator $\bar{h}_i^j$ of $\Sigma_t$ in $(\mathbb{R}^n\setminus D_{r_0}, \delta_{ij})$ evolves by
\begin{align*}
  \frac{\pt }{\pt t}\bar{h}_i^j =& ~-\nabla^j\nabla_i \left(\frac 1{\bar{H}+d\psi\cdot \bar{\nu}}\right)-\frac {\bar{h}_i^k\bar{h}_k^j}{\bar{H}+d\psi\cdot \bar{\nu}}\\
  = & \frac{\nabla^j\nabla_i (d\psi\cdot\bar{\nu})}{\left(\bar{H}+d\psi\cdot \bar{\nu}\right)^2}-\frac{2\nabla^j(d\psi\cdot\bar{\nu})\nabla_i (d\psi\cdot\bar{\nu})}{\left(\bar{H}+d\psi\cdot \bar{\nu}\right)^3}-\frac {\bar{H}^2\delta_i^j}{(n-1)^2\left(\bar{H}+d\psi\cdot \bar{\nu}\right)}
\end{align*}
As in \eqref{s5:conf2},
\begin{equation*}
  d\psi\cdot\bar{\nu}=~-\frac {mr^{1-n}}{1+\frac m2r^{2-n}}\pt_r\cdot \bar{\nu}=v(|X|^2)X\cdot \bar{\nu},
\end{equation*}
where $|X|^2=r^2$, $X\cdot\bar{\nu}=r\pt_r\cdot\bar{\nu}$ and $v(\cdot):\mathbb{R}^+\to \mathbb{R}$ is a function given by
\begin{equation*}
  v(x):=~-\frac{mx^{-n/2}}{1+\frac m2x^{1-\frac n2}}.
\end{equation*}
Then
\begin{equation*}
  \nabla_i(d\psi\cdot\bar{\nu})=~\left(2v' \langle X,\bar{\nu}\rangle+\frac{\bar{H}}{n-1}v\right)\langle X,e_i\rangle
\end{equation*}
and
\begin{align*}
  \nabla^j\nabla_i(d\psi\cdot\bar{\nu})=&~4\left(v''\langle X,\bar{\nu}\rangle+\frac{\bar{H}}{n-1}v'\right)\langle X,e_i\rangle \langle X,e_j\rangle\\
  &\quad +\left(2v'\langle X,\bar{\nu}\rangle+\frac{\bar{H}}{n-1}v\right)\left(1-\frac{\bar{H}}{n-1}\langle X,\bar{\nu}\rangle\right)\delta_i^j.
\end{align*}
Since $\Sigma_t$ is totally umbilic for  $t\in [0,\delta)$, the trace-less second fundamental form $\mathring{\bar{h}}_i^j$ is zero for all time $t\in [0,\delta)$. Then
\begin{align*}
  0=&~\frac{\pt }{\pt t}\mathring{\bar{h}}_i^j =~\frac{\pt }{\pt t}\left(\bar{h}_i^j-\frac{\bar{H}}{n-1}\delta_i^j\right)\displaybreak[0]\\
  = & \left(\frac{4\left(v''\langle X,\bar{\nu}\rangle+\frac{\bar{H}}{n-1}v'\right)}{\left(\bar{H}+d\psi\cdot \bar{\nu}\right)^2}-\frac{2\left(2v'\langle X,\bar{\nu}\rangle+\frac{\bar{H}}{n-1}v\right)^2}{\left(\bar{H}+d\psi\cdot \bar{\nu}\right)^3}\right)\displaybreak[0]\\
  &\qquad \times \left(\langle X,e_i\rangle \langle X,e_j\rangle-\frac{|X^{\top}|^2}{n-1}\delta_i^j\right),
\end{align*}
where $X^{\top}$ denotes the tangential part of the position vector. It follows that $\langle X,e_i\rangle^2={|X^{\top}|^2}/{(n-1)}$ on $\Sigma_t$ and is independent of the direction $e_i$. This can occur only if position vector $X$ is parallel to the normal vector at $X$ and each $\Sigma_t, t\in [0,\delta),$ is a sphere centered at the origin. Therefore, each $\Sigma_t, t\in [0,\delta),$ is a slice $\{s\}\times \mathbb{S}^{n-1}$ in the Schwarzschild space. This completes the proof of Theorem \ref{main-thm}.
\end{proof}

\bibliographystyle{amsplain}

\begin{thebibliography}{99}

\bibitem{Bray-M}H. Bray and P. Miao, \emph{On the capacity of surfaces in manifolds with
nonnegative scalar curvature}. Invent. Math. \textbf{172}(2008), 459--475.

\bibitem{BN}H. Bray and A. Neves, Classification of prime 3-manifolds with Yamabe invariant
greater than $\mathbb{RP}^3$, Ann. of Math. \textbf{159}(2004), 407--424.

\bibitem{BHW} S. Brendle, P.-K. Hung and M.-T. Wang, {\it A Minkowski-type inequality for hypersurfaces in the Anti-deSitter-Schwarzschild manifold}, Comm. Pure Appl. Math., \textbf{69}(2016), no. 1, 124--144.

\bibitem{Ba-M82}E. Barozzi, U. Massari, \emph{Regularity of minimal boundaries with obstacles}, Rendiconti del Seminario Matematico della Universita di Padova, \textbf{66}(1982), 129--135.

\bibitem{ES}L.C. Evans and J. Spruck, \emph{Motion of level--sets by mean curvature, I}, J. Differential
Geom. \textbf{33} (1991), 635--681.

\bibitem{Fre-Sch-2014}A. Freire and F. Schwartz, \emph{Mass-capacity inequalities for conformally flat manifolds with boundary}, Comm. PDE, \textbf{39}(2014), no.1, 98--119.

\bibitem{GL} P. Guan and J. Li, {\it The quermassintegral inequalities for k-convex starshaped domains,} Adv. Math. \textbf{221}(2009), 1725--1732.

\bibitem{Hei01} M. Heidusch, \emph{Zur Regularit\"{a}t des inversen mittleren Kr\"{u}mmungsflusses,} PhD
thesis, Eberhard-Karls-Universit\"{a}t T\"{u}bingen (2001).

\bibitem{Hui09}G. Huisken, \emph{Inverse mean curvature flow and isoperimetric inequalities}, video available at
\url{https://video.ias.edu/node/233} (2009).

 \bibitem{HI01} G. Huisken and T. Ilmanen,  {\it The inverse mean curvature flow and the Riemannian Penrose inequality,} J. Differential Geom. \textbf{59} (2001), 353--438.

%\bibitem{HI08} G. Huisken and T. Ilmanen, {\it Higher regularity of the inverse mean curvature flow}, J. Differential Geom., \textbf{80}(2008), 433--451.

\bibitem{Il94}   T. Ilmanen, \emph{ Elliptic regularization and partial regularity for motion by mean curvature}, Mem. Amer. Math. Soc. \textbf{108}(520) (1994).

\bibitem{Lee-Neves} Dan A. Lee and A. Neves, \emph{The Penrose inequality for asymptotically locally hyperbolic spaces with nonpositive mass}, Commun. Math. Phys.£¬ \textbf{339}(2015), 327--352 .

\bibitem{Li-Wei-Schwar}H. Li and Y. Wei, \emph{On inverse mean curvature flow in Schwarzschild space and Kottler space}, 	Calc. Var.,  (2017) \textbf{56}: 62.

\bibitem{Moore}K. Moore, \emph{On the evolution of hypersurfaces by their inverse null mean curvature}, J. Differential Geom., \textbf{98}(2014), no.3, 425--466.

\bibitem{Mar}T. Marquardt, \emph{Weak solutions of inverse mean curvature flow
for hypersurfaces with boundary},  J. Reine Angew. Math. \textbf{728 }(2017), 237--261.

%\bibitem{Roth-Sch}J. Roth and J. Scheuer, \emph{Explicit rigidity of almost-umbilical hypersurfaces}, arXiv:1504.05749,  to appear in Asian J. Math.

\bibitem{Sch08}F. Schulze, \emph{Nonlinear evolution by mean curvature and isoperimetric inequalities}, J. Differential Geom., \textbf{79}(2008), no.2, 197--241.

\bibitem{sch2017}J. Scheuer, \emph{The inverse mean curvature flow in warped cylinders of non-positive radial curvature}, Adv. Math. \textbf{306}(2017), 1130--1163.

\bibitem{LSim} L.M. Simon, \emph{Lectures on Geometric Measure Theory,} Proc. of the Centre for
Math. Analysis, 3, Austr. Nat. Univ., 1983.


\bibitem{Ta82} I. Tamanini, \emph{Boundaries of Caccioppoli sets with H\"{o}lder-continuous normal
vector}, J. Reine Angew. Math., \textbf{334} (1982), 27--39.
\end{thebibliography}
%-------------------------------------------------

%%-------------------------------------------------
\end{document}